\documentclass{article}
\usepackage{amsmath, amsthm, epsfig,amssymb, multicol, tikz}
\newtheorem{lemma}{Lemma}
\newtheorem{theorem}{Theorem}

\newtheorem{definition}{Definition}

\usepackage{rotating}
\usepackage{tikz}
\usetikzlibrary{topaths,calc}

\tikzstyle{vertex} = [fill,shape=circle,node distance=80pt]
\tikzstyle{edge} = [fill,opacity=.5,fill opacity=.5,line cap=round, line join=round, line width=50pt]
\tikzstyle{elabel} =  [fill,shape=circle,node distance=30pt]

\pgfdeclarelayer{background}
\pgfsetlayers{background,main}

\title{A Bound on the Spectral Radius of Hypergraphs with $e$ Edges}
\author{
Shuliang Bai \thanks{University of South Carolina, Columbia, SC 29208,
({\tt sbai@math.sc.edu}).} \and 
Linyuan Lu
\thanks{University of South Carolina, Columbia, SC 29208,
({\tt lu@math.sc.edu}). This author was supported in part by NSF
grant DMS 1300547 and DMS 1600811.}
}
\begin{document}
\maketitle
\begin{abstract}
For $r\geq 3$, let
 $f_r\colon [0,\infty)\to [1,\infty)$ be the unique analytic function such that
$f_r({k\choose r})={k-1\choose r-1}$ for any $k\geq r-1$.  We prove that the spectral radius of an 
$r$-uniform hypergraph $H$ with $e$ edges is at most $f_r(e)$. The equality holds
if and only if $e={k\choose r}$ for some positive integer $k$ and $H$ is the union of a complete $r$-uniform hypergraph 
$K_k^r$ and some possible isolated vertices. 
This result generalizes the classical Stanley's theorem
on graphs.
\end{abstract}


\textsl{keywords:} Spectral radius, Uniform hypergraph, Adjacency tensor, $\alpha$-normal labeling, Stanley's theorem\\

\section{History}
The spectral radius $\rho(G)$ of a graph $G$ is the maximum eigenvalue of its adjacency matrix. Which graph has the maximum spectral radius among all graphs with $e$ edges?
If $e={k \choose 2}$,  Brualdi and Hoffman \cite{Brualdi} proved that the maximum of $\rho(G)$ is reached by the union of a compete graph on $k$ vertices and some possible isolated vertices. They conjectured that 
the maximum spectral radius of a graph $G$ with $e={k\choose 2} + s$ edges is attained by the graph
$G_e$, which is obtained from complete graph $K_k$ by adding a new vertex and $s$ new edges.
In 1987, Stanley \cite{stanley} proved that
the spectral radius of a graph $G$ with $e$ edges
is at most $\frac{\sqrt{1+8e}-1}{2}$. The equality holds if and only if $e={k\choose 2}$
and $G$ is the union of the complete graph $K_k$ and some isolated vertices.
Friedland \cite{Friedland} proved a bound which is tight on 
the complete graph with one, two, or three edges removed or the complete graph with one edge added. 
Rowlinson \cite{Row} finally confirmed Brualdi and Hoffman's conjecture, and proved that 
$G_e$ attains the maximum spectral radius among all graphs with $e$ edges.  

On the problem of maximizing spectral radius of a certain class of hypergraphs, 
Fan, Tan, Peng and Liu \cite{FTPL} determined the extremal spectral radii of several classes of $r$-uniform hypergraphs with few edges. 
Xiao,  Wang and Lu \cite{XWL} determined the unique $r$-uniform supertrees  with maximum spectral radii among all $r$-uniform supertrees with given degree sequences.  
Li, Shao, and Qi \cite{LSQ} determined the extremal spectral radii of $r$-uniform supertrees.
In \cite{KLQY}, Kang, Liu, Qi, and Yuan 
solved a conjecture of Fan et al.\cite{FTPL} related to compare the spectral radii of some $3$-uniform hypergraphs. Chen, Chen, and Zhang \cite{CCZ} proved several good upper bounds for the adjacency and signless Laplacian spectral radii of uniform hypergraphs in terms of degree sequences.

In this paper, we will generalize Stanley's theorem to hypergraphs, that is, maximizing the spectral radius of $r$-uniform hypergraphs among all $r$-uniform hypergraphs with a given number of edges. 
For $r\geq 3$, an $r$-uniform hypergraph $H$ on $n$ vertices consists of a vertex set $V$ and an edge set  $E\subseteq {V\choose r}$.
The adjacency tensor $A$ of $r$-uniform hypergraph $H$ refers to an  $r$-order $n$-dimensional tensor
$A=(a_{i_1 \cdots i_r})$ defined by
$$a_{i_1 \cdots i_r}=
 \begin{cases}
\frac{1}{(r-1)!}  & \text{if $i_1\cdots i_r$ is an edge of $H$ ,} \\
0  & \text{otherwise,}
\end{cases}
$$
where each $i_j$ runs from $1$ to $n$ for $j\in[r]$. 
The adjacency tensor $A$ of $r$-uniform hypergraph is always  nonnegative and symmetric.
  
Given a $r$-uniform hypergrpah $H$, the polynomial form $P_H(\mathbf{x})\colon \mathbb{R}^n\to \mathbb{R}$  is defined for any vector  $\mathbf{x}=\left( x_1, \cdots, x_n\right) \in \mathbb{R}^n$ as 
$$ P_H(\mathbf{x})= \sum\limits_{i_1 \cdots i_r=1}^n a_{i_1 \cdots i_r} x_{i_1} \cdots x_{i_r}= r\sum\limits_{i_1 \cdots i_r\in E(H)}x_{i_1} \cdots x_{i_r}.$$
Then the spectral radius of a $r$-uniform hypergrpah $H$ is
$$\rho(H)=\max_{ \Vert\mathbf{x}\Vert_r=1} \rho(H)= \max_{ \Vert\mathbf{x}\Vert_r=1}r\sum\limits_{i_1 \cdots i_r\in E(H)} x_{i_1}\cdots x_{i_r}, 
$$
where $\Vert\mathbf{x}\Vert_r=\left(\sum\limits_{i=1}^{n} |x_i|^r\right)^{1/r}.$

In general, one can also define the spectral radius of any tensor $A$ using eigenvalues.
A pair $(\lambda, \mathbf{x})\in \mathbb{C}\times (\mathbb{C}^n\setminus\{\mathbf{0}\})$  is called an eigenvalue and an eigenvector of $A$ if they satisfy $A\mathbf{x}^{r-1}=\lambda \mathbf{x}^{[r-1]}$, that is for any $i\in [n]$, 
$$\sum\limits_{i_2,...,i_r=1}^{n} a_{i i_2 \cdots i_r} x_{i_2} \cdots x_{i_r}=\lambda x_{i}^{r-1}. $$
%where  $\mathbf{x} = (x_1,\cdots , x_n)\in \mathbb{C}^n\setminus\{\mathbf{0}\}$. 
The spectral radius $\rho(A)$ is defined to be the largest modulus of 
eigenvalues of $A$.
When $A$ is symmetric and non-negative, 
the two definitions are equivalent.
If $\mathbf{x}$ is a real eigenvector of $A$, clearly the corresponding
eigenvalue $\lambda$ is also real. In this case, $\mathbf{x}$ is called an {\em $H$-eigenvector} and $\lambda$  an {\em $H$-eigenvalue}.
Furthermore, if $\mathbf{x}\in\mathbb{R}^n_+$,
where $\mathbb{R}^n_+ = \{x \in \mathbb{R}^n : x \geq 0\}$, 
then $\lambda$ is an {\em $H^+$-eigenvalue} of $A$. 
If $\mathbf{x}\in \mathbb{R}^n_{++}$, 
where $\mathbb{R}^n_{+ +}= \{x \in \mathbb{R}^n : x > 0\}$,
then $\lambda$ is said to be
an {\em $H^{++}$-eigenvalue} of $A$. 
\begin{theorem}
\textbf{(Perron-Frobenius theorem for non-negative tensors)}
\begin{enumerate} 
\item (Yang and Yang 2010 \cite{YangYang}) If $A$ is nonnegative tensor of order $r$ and dimension $n$, 
then the spectral radius $\rho(A)$ is an $H^+$-eigenvalue of $A$. 
\item (Frieland, Gaubert and Han 2011 \cite{FriedlandGaubetHan}) If furthermore $A$ is weakly irreducible, 
then  $\rho(A)$ is the unique  $H^{++}$-eigenvalue of $A$, 
with the unique eigenvector $\mathbf{x}\in \mathbb{R}^n_{++}$ , up to a positive scaling coefficient.
\item (Chang, Pearson and Zhang 2008 \cite{ChangPearsonZhang}) If moreover $A$ is irreducible, then  $\rho(A)$ is the unique  $H^{+}$-eigenvalue of $A$, 
with the unique eigenvector $\mathbf{x}\in \mathbb{R}^n_{+}$, up to a positive scaling coefficient.
\end{enumerate}
\end{theorem}
 
Perason and Zhang\cite{PerasonZhang} proved that 
the adjacency tensor $A$ of a connected hypergraph $H$ 
is weakly irreducible, 
thus by Perron-Frobenius theorem, 
there exists a unique positive eigenvector up to scales 
corresponding to $\rho(H)$. 
And this eigenvector is called {\em Perron-Frobenius vector}. 
Please read the survey paper \cite{CQZ}  on the spectral theory of nonnegative tensors for the terminologies not defined in this paper.

Note that the spectral radius of the complete hypergraph $K_n^r$ is ${n-1\choose r-1}$.
This motivated us to define an analytic function $f_r\colon [0,\infty)\to [1,\infty)$ so that
\begin{equation}
  \label{eq:fr}
f_r\left({n\choose r}\right)={n-1\choose r-1}.  
\end{equation}
See the detailed definition of $f_r(x)$ in section 3. 
The following theorem generalizes Stanley's theorem.

\begin{theorem}\label{Th1}
  For $r\geq 2$, suppose that $H$ is an $r$-uniform hypergraph with $e$ edges.
Then its spectral radius $\rho(H)$ is at most $f_r(e)$.
The equality holds if and only if 
$e={k\choose r}$ for an integer $k$ and $H$ is the complete $r$-uniform
hypergraph $K_k^r$ possibly with some isolated vertices added.
\end{theorem}

Note that $f_2(x)$ satisfies $f_2({n\choose 2})=n-1$. Let $e={n\choose 2}$ and solve
for $n$. We get
\begin{equation}
  \label{eq:f2}
f_2(e)=\frac{\sqrt{8e+1}-1}{2}.  
\end{equation}

Stanley's theorem is just a special case with $r=2$.

The main tool that we used in the paper is
the $\alpha$-normal labeling method, which was first developed by the second author and Dr. Man to classifying all connected
$r$-uniform hypergraphs with spectral radius at most $\sqrt[r]{4}$ in the paper \cite{LuMan}.
This method is used in \cite{KLQY}
and is generalized in \cite{ZLKB}.

The paper is organized as follows: In section 2, we review tools and prove new lemmas  regarding 
the spectral radius of $r$-uniform hypergraphs.
An important lemma about the function $f_r(x)$ will be proved in section 3. Finally we prove our main theorem in the last section.

\section{Lemmas on uniform hypergraphs}
Let $H=(V, E)$ be a connected $r$-uniform hypergraph whose spectral radius attains the maximum among all
 the $r$- uniform hypergraphs with $e$ edges. We call $H$ a maximum hypergraph. 
%Let $\rho(H)$ be the spectral radius, $\textbf{x}$ be the Perron-Frobenius vector of $H$. 

For $k\geq 1$, moving $k$ edges $(e_1, \cdots, e_k)$ from $(v_1,\cdots , v_k)$ to $v$
means replacing each edge $e_i$ by new edge $e_i\setminus \{v_i\}\cup \{v\}$ for $i=1,2,\ldots,k$.
Here $v_i$ is a vertex incident to $e_i$. 
This edge-shifting operation can be used to increase the spectral radius.

\begin{lemma}\label{makelarger}
\cite{LSQ} 
Let $k \geq 1$ and let $H$ be a connected $r$-hypergraph. Let $H'$
be the hypergraph obtained from $H$ by moving edges $(e_1, \cdots, e_k)$ from $(v_1,\cdots , v_k)$ to $v$.
Assume that $H′$ contains no multiple edges. If $\mathbf{x}$ is a Perron vector of $H$ and
$x_v \geq \max_{1\leq i \leq k} x_{v_i},$ then $\rho(H') > \rho(H)$.
\end{lemma}

By Lemma \ref{makelarger},  we can increase the spectral radius by doing the edge-shifting operations stated in the lemma; this process will end until no potential edge
can be moved in $H$. The resulted graph only has one non-trivial connected component.

\begin{lemma}\label{commonvertex}
\cite{FTPL}
If $H$ is a maximum hypergraph among the connected hypergraphs with fixed number edges, then $H$ contains a vertex $v$ adjacent to all other vertices. 
\end{lemma}
%\begin{proof}
%Let $H$ be the maximum hypergraph among the connected hypergraphs with fixed number edges. 
%Let $\textbf{x}$ be the Perron-Frobenius vector of $H$ and 
%$v$ be a vertex such that $x_v = \max\{ x_u, u\in V(G)\}$.
%Suppose vertex $u$ is not adjacent to $v$, 
%then there exists an edge $e\in E(H)$ such that $u, w\in e$, but $v\notin e$. By the edge-shifting operation, 
%we can remove edge $e$ from $w$ to $v$. By lemma \ref{makelarger}, 
%we obtain a new hypergraph with same number edges
%but larger spectral radius, a contradiction. 
%\end{proof}

{\bf Remark:} In fact, from the proof of the above Lemma, one can choose $v$ to be any one of the vertices where
the Perron-Frobenius vector achieves the maximum value.

\begin{definition}\label{shadef}
Given a family $\mathcal{F}$ of $r$-sets, the shadow $\partial(\mathcal{F})$ is defined as 
$$\partial(\mathcal{F})=\{e': e'=e\setminus\{v\}, \text{for some $e\in\mathcal{F}$, and $v\in e$} \}.$$
\end{definition}

\begin{definition}\label{link}
Given an $r$-hypergraph $H$ and a vertex $v$ of $H$, the link graph $G_v$ is the $(r-1)$-graph consisting of all $S \subset V (H)$ with $|S| = r - 1$
and $S \cup \{v\} \in E(H)$. 
\end{definition}

The celebrated Kruskal-Katona Theorem 
determines the minimum size of the shadow $\partial(\mathcal{F})$ given
the size of $\cal F$.

\begin{theorem}
(Kruskal \cite{Kruskal} and Katona \cite{Katona}) Any $r$-uniform set family $\mathcal{F}$ of size
$m={a_r\choose r }+{a_{r-1}\choose r-1}+\cdots+{a_k\choose k}$, where $a_r > a_{r-1} > \cdots > a_k \geq k \geq 1,$ must have
$$|\partial(F)|\geq {a_r\choose r-1 }+{a_{r-1}\choose r-2}+\cdots+{a_k\choose k-1}.$$
\end{theorem}

Kruskal-Katona Theorem has many applications. However, it is not easy to apply directly. In this paper,
we use a slightly weaker version due to Lov\'asz : 

\begin{theorem}\label{shadow} 
(Lov\'asz \cite{Lovas}) Any $r$-uniform set family $\mathcal{F}$ of size $m =
{x\choose r}$ where $x$ is a real and $x \geq r$, must have
$$|\partial(F)|\geq {x \choose r-1}.$$
\end{theorem}

%Let $d$ be the degree of the vertex $v$ in $H$. The degree is the number of edges in $H$ that containing $v$. 
In a hypergraph $H$, the degree $d(v)$ of a vertex $v$ is the number of edges that contain $v$.
Let $H_v$ be the induced subgraph obtained from $H$ by deleting the vertex $v$. 
Let $G_v$ be the link graph of $v$.  
By definition of $G_v$, $d(v)$ is also the number of edges in $G_v$. 
We have the following lemma:
%Lemma \ref{makelarger} implies the following straight forward corollary: 

\begin{lemma}
 \label{GvHv}
Suppose that a connected hypergraph $H$ reaches 
the maximum spectral radius $\rho(H)$ among all $r$-uniform hypergraphs with $e$ edges.
Suppose that the Perron-Frobenius vector of $H$ reaches the maximum at a vertex $v$.
Then we have the following properties:
\begin{enumerate}
\item The shadow graph $\partial(H_v)$ of $H_v$ is a subgraph of the link graph $G_v$.
\item The link graph $G_v$ is connected while $H_v$ may be disconnected but
has only one non-trivial connected component.

\item The link graph $G_v$ has at least $f_r(e)$ edges. 
%, thus degree of $v$ satisfies $d(v)\geq f_r(e) $.
\end{enumerate}
\end{lemma}
\begin{proof}
Let $H$ be the maximum hypergraph on vertices $v_1, v_2, \cdots, v_n$.
Let $\textbf{x}=(x_1, x_2, \cdots, x_n)$ be the Perron vector, 
in which each $x_i$ is the entry of $\textbf{x}$ corresponding to vertex $v_i$ for $i=1, 2, \cdots, n$.  Then we have $x_v\geq x_u$ for any other vertex $u$.

We will prove Item 1 by contradiction.
Suppose $\partial(H_v)$ is not a subgraph of $G_v$. 
Then there exists an $(r-1)$-subset 
$\{v_{i_1},\ldots, v_{i_{r-1}}\}$ in $\partial(H_v)$ but not in $E(G_v)$.
By the definition of the shadow $\partial(H_v)$, there is a vertex $u\not= v$
so that $\{u,v_{i_1},\ldots, v_{i_{r-1}}\}$ is an edge of $H_v$. By moving this edge
from $u$ to $v$, we obtain a new hypergraph $H'$ from $H$ with larger spectral radius
as guaranteed by Lemma \ref{makelarger}, a contradiction.

For Item 2, removing all edges of $H_v$ from $H$, the resulted hypergraph is still connected.
Thus $G_v$ has no isolated vertices. Now we will prove that $G_v$ is connected.
Otherwise, $G_v$ has at least two non-trivial connected components.
Let $\{v_{j_1}, \cdots,  v_{j_{r-1}}\}$ and  $\{v_{k_1}, \cdots,  v_{k_{r-1}}\}$
 be any two edges  from different connected components of $G_v$.
Here we assume the vertices are ordered non-increasingly according to the Perron-Fronenius
vector $\bf x$; that is $x_{j_1}\geq x_{j_2}\geq\cdots\geq x_{j_{r-1}}$ and $x_{k_1}\geq x_{k_2}\geq\cdots\geq x_{k_{r-1}}$. We also assume that $x_{j_1}\geq x_{k_1}$.
By Lemma \ref{makelarger},
We can move the edge $\{v, v_{k_1}, \cdots,  v_{k_{r-1}}\}$ from $v_{k_1}$ to $v_{j_1}$ to increase
the spectral radius. Contradiction.
A similar argument can show that $H_v$ has only one non-trivial component.

For Item 3, we write $e={s\choose r}$
and $|E(H_v)|={y \choose r}$ for some real numbers $s, y\geq r-1$.
By Theorem \ref{shadow}, we have 
$$|\partial(H_v)|\geq {y \choose r-1}.$$
Applying Lemma \ref{GvHv}, we have 
$$|E(G_v)|\geq |\partial(H_v)|\geq
 {y \choose r-1}.$$ 
Thus,
\begin{align*}
{s\choose r}&=|E(H)| \\
&=|E(G_v)|+|E(H_v)|\\
&\geq  {y \choose r-1} +{y \choose r}\\
&={y+1\choose r}. 
\end{align*}
Thus, $s\geq y+1$. It implies 
\begin{align*}
|E(G_v)|&=e-|E(H_v)|\\
&={s\choose r}-{y\choose r} \\
&\geq {s\choose r}-{s-1\choose r}\\
&={s-1\choose r-1}\\
&=f_r(e).
\end{align*}
The proof is finished.
\end{proof}

In \cite{LuMan}, Lu and Man discovered a novel way to link the spectral radius to 
$\alpha$-normal labeling of any connected hypergraph.
\begin{definition}
\cite{LuMan}
 A weighted incidence matrix $B$ of a hypergraph $H$ is a 
$|V | \times |E|$ matrix such that for any vertex $v$ and any edge $e$, the entry $B(v, e) > 0$ if $v\in e$ and $B(v, e) = 0$ if $ v \notin e$.
\end{definition}

\begin{definition}
\cite{LuMan}
 \begin{enumerate}
 \item 
A hypergraph $H$ is called $\alpha$-normal if there exists a weighted incidence matrix $B$ satisfying
\begin{enumerate}
 \item $\sum_{e:v \in e} B(v, e)= 1$, for any $v\in V(H)$. 
 \item $\prod_{v \in e} B(v, e)=\alpha$, for any $e\in E(H)$. 
 \end{enumerate}
Moreover, the incidence matrix $B$ is called consistent if for any cycle $v_0e_1v_1\cdots v_l (v_l=v_0)$, 
$$ \prod\limits_{i=1}^{l} \frac{B(v_i, e_i)}{B(v_{i-1}, e_i)}=1.$$
In this case, $H$ is called consistently $\alpha$-normal.

\item  A hypergraph $H$ is called $\alpha$-subnormal if there exists a weighted incidence matrix $B$ satisfying
 \begin{enumerate}
 \item $\sum_{e:v \in e} B(v, e)\leq 1$, for any $v\in V(H)$. 
 \item $\prod_{v \in e} B(v, e)\geq\alpha$, for any $e\in E(H)$. 
 \end{enumerate}
 Moreover, $H$ is called strictly $\alpha$-subnormal if it is $\alpha$-subnormal but not $\alpha$-normal. 
  \end{enumerate}
\end{definition}
Definitions about $α\alpha$-supernormal hypergraph is defined in \cite{LuMan}, but we omit it since it is irrelevant. 

\begin{lemma}\label{connectedgraph}
\cite{LuMan}
Let $H$ be a connected $r$-uniform hypergraph. Then the spectral radius of $H$ is $\rho(H)$ if and only if $H$ is consistently $\alpha$-normal with $\alpha = \left(1/\rho(H)\right)^r.$
\end{lemma}

\begin{lemma} \label{subtheorem}
\cite{LuMan} Let $H$ be an $r$-uniform hypergraph. 
\begin{enumerate}
\item If $H$ is consistently $\alpha$-normal, then the spectral radius of H
satisfies
$$\rho(H)=\alpha^{-\frac{1}{r}}. $$
\item If $H$ is $\alpha$-subnormal, then the spectral radius of H
satisfies
$$\rho(H)\leq \alpha^{-\frac{1}{r}}. $$
\end{enumerate}

\end{lemma}

\section{Lemmas on the function $f_r(x)$}
For a fixed positive integer $r$, consider the polynomial $p_r(x)=\frac{x(x-1)\cdots (x-r+1)}{r!}.$
Since the binomial coefficient ${n\choose r}=p_r(n)$, we view ${x\choose r}$
as the polynomial $p_r(x)$.
Note that $p_r(x)$ is an increasing function over the interval $[r-1,\infty)$
so that the inverse function exists. Let $p^{-1}_r\colon [0,\infty) \to
[r-1, \infty)$ denote the inverse function of
$p_r(x)$ (when restricted to the interval $[r-1, \infty)$.
We define a function $f_r\colon [0,\infty) \to [1,\infty)$ as follows:
$$f_r(x):= p_{r-1}(p_{r}^{-1}(x)-1).$$
Thus $f_r(x)$ satisfies Equation (\ref{eq:fr}). 
This function plays an essential role in this paper. It has the following properties.

\begin{lemma}\label{l1}
  Suppose $y={\eta \choose r}$ with $\eta\geq r-1$. Then we have
  \begin{enumerate}
  \item $\eta=\frac{ry}{f_r(y)}$.
  \item $f_r(y)$ is an increasing function on $[0,\infty)$.
  \item The derivative of $f_r(y)$ is given by:
$$f_r'(y)=\frac{r}{\eta}\frac{\sum_{i=1}^{r-1}\frac{1}{\eta-i}}{\sum_{i=0}^{r-1}\frac{1}{\eta-i}}.$$
  \end{enumerate}
\end{lemma}

\begin{proof}
The formula $f_r(y)={\eta-1 \choose r-1}= \frac{r}{\eta} {\eta \choose r}$ implies item 1. 
Note $$\ln p_r(x)=\sum_{j=0}^{r-1}\ln(x-j) -\ln (r!).$$
We have
$$p_r(x)'= p_r(x) \cdot \frac{d}{dx}\left[\sum_{j=0}^{r-1}\ln(x-j) -\ln (r!)\right]
= {x \choose r}\sum_{j=0}^{r-1} \frac{1}{x-j} .$$

View $\eta$ as a function of $y$ and apply the Chain rule. We have
\begin{align*}
f_r'(y)={\eta-1 \choose r-1}'
=  \frac{df_r}{d\eta}\frac{d\eta}{dy}
= \frac{\frac{df_r}{d\eta}}{\frac{dy}{d\eta}}
= \frac{{\eta-1 \choose r-1}\sum^{r-1}_{i=1} \frac{1}{\eta-i}}{{\eta \choose r}\sum^{r-1}_{i=0} \frac{1}{\eta-i}}
= \frac{r}{\eta}\frac{\sum^{r-1}_{i=1} \frac{1}{\eta-i}}{\sum^{r-1}_{i=0} \frac{1}{\eta-i}}.
\end{align*}
Since $\eta>r-1$,
the right-hand side of $f_r'(y)$ is positive. Thus $f_r(y)$ is an increasing function. 
\end{proof}

For convenience, we also define $f_1$ to be the constant function $f_1(x)\equiv 1$.
\begin{lemma}\label{l2}
For an integer $r\geq 2$ and any two reals $e$ and $x$ with $e\geq x\geq f_r(e)$,  we have
\begin{equation}
  \label{eq:F}
  \frac{x^{1/(r-1)} f_{r-1}(x)}{f_r(e)^{r/(r-1)}} + \frac{f_r(e-x)}{f_r(e)}\leq 1.
\end{equation}
\end{lemma}
\begin{proof}
Let $F(x)=\frac{x^{1/(r-1)} f_{r-1}(x)}{f_r(e)^{r/(r-1)}} + \frac{f_r(e-x)}{f_r(e)}$. Note $F(x)$ is a smooth function. 
To show $F(x)\leq 1$ for all $x\in [f_r(e),e]$, it is sufficient to prove the following facts:
\begin{enumerate}
\item $F(f_r(e))=1$. \label{item1}
\item $F'(f_r(e))<0$. \label{item2}
\item $F''(x)\leq 0$ for any $x\in [f_r(e), e]$.\label{item3}
\end{enumerate}

Note that item 3 indicates that $F'(x)$ is a decreasing function 
on $[f_r(e), e)$, together with item 2,  we get $F'(x) < 0$,  implies that $F(x)$ is a strictly decreasing function on $[f_r(e), e]$.   
By item 1, we conclude that $F(x)\leq 1$ for $x\in [f_r(e), e]$, with the inequality holds if and only if $x=f_r(e)$.

Let $s$, $t$, and $u$ be three positive reals satisfying
$e={s\choose r}$, $x={t\choose r-1}$, and $e-x={u\choose r}$.
Since $x={t\choose r-1}\geq f_r(e)={s-1\choose r-1}$,
we have $t\geq s-1$. 
Similarly, we have $${u\choose r}= e-x ={s\choose r}-{t\choose r-1}\leq {s\choose r}-{s-1\choose r-1}
={s-1\choose r}.$$ 
It implies that $u\leq s-1$. 
Thus, we have
\begin{equation}
  \label{eq:tsu}
  t\geq s-1\geq u,
\end{equation}
with the equality holds if and only if $x=f_r(e)={s-1\choose r-1}$.

% Note when $x=f_r(e)={s-1\choose r-1}$, ${u\choose r}=e-x={s-1\choose r}$, 
% we have $t=u=s-1$. 
% In $F(x)$,  if we replace $x$ by $f_r(e)={s-1 \choose r-1}$,  
% then $f_{r-1}(x)={s-2 \choose r-2}$, and $f_r(e-x)=f_{r}({s-1 \choose r})={s-2\choose r-1}$, so we have 
We have
\begin{align*}
F(f_r(e))
& = \frac{{s-1 \choose r-1}^{1/(r-1)} {s-2 \choose r-2}}{{s-1\choose r-1}^{r/(r-1)}} + \frac{{s-2\choose r-1}{s-1\choose r-1}^{1/(r-1)}}{{s-1\choose r-1}^{r/(r-1)}} \\
& = \frac{{s-1 \choose r-1}^{1/(r-1)} \left[{s-2 \choose r-2}+{s-2\choose r-1}\right]}{{s-1\choose r-1}^{r/(r-1)}} \\
& = \frac{{s-1 \choose r-1}^{1/(r-1)} {s-1 \choose r-1}}{{s-1\choose r-1}^{r/(r-1)}}\\
& =1.
\end{align*}
Proof of item 1 is finished. 

Now we compute the derivative of $F(x)$. Note that $e$ and $s$ are constants
while $t$ and $u$ are functions of $x$.
Applying item 1 of Lemma \ref{l1} to $x={t\choose r-1}$ 
and $x=-{u\choose r}+e$, 
we get
\begin{align*}
  f_{r-1}'(x) &=\frac{(r-1)\sum_{i=1}^{r-2}\frac{1}{t-i}}{t\sum_{i=0}^{r-2}\frac{1}{t-i}},\\
f_r'(e-x) &= 
-\frac{r\sum_{i=1}^{r-1}\frac{1}{u-i}}{u\sum_{i=0}^{r-1}\frac{1}{u-i}}.
\end{align*}
% Also we have
% \begin{align}
%   f_{r-1}(x)&=\frac{x(r-1)}{t}
%  %f_r(e-x)&=\frac{r(e-x)}{u}.
% \end{align}
Thus, we have
 \begin{align*} \label{eq:F'}
F'(x)= &\frac{\frac{1}{r-1}x^{1/(r-1)-1}f_{r-1}(x)+x^{1/(r-1)}\frac{(r-1)\sum_{i=1}^{r-2}\frac{1}{t-i}}{t\sum_{i=0}^{r-2}\frac{1}{t-i}}
 }{f_r(e)^{r/(r-1)}} - \frac{r\sum_{i=1}^{r-1}\frac{1}{u-i}}{f_r(e)u\sum_{i=0}^{r-1}\frac{1}{u-i}}\\[2mm]
&= \frac{\frac{1}{r-1}x^{1/(r-1)-1}\frac{x(r-1)}{t}+x^{1/(r-1)}\frac{(r-1)\sum_{i=1}^{r-2}\frac{1}{t-i}}{t\sum_{i=0}^{r-2}\frac{1}{t-i}}}{f_r(e)^{r/(r-1)}} - \frac{r}{f_r(e)}\frac{\sum_{i=1}^{r-1}\frac{1}{u-i}}{u\sum_{i=0}^{r-1}\frac{1}{u-i}}\\[2mm]
& = \frac{x^{1/(r-1)}}{f_r(e)^{r/(r-1)}}\left(\frac{1}{t}+\frac{(r-1)\sum_{i=1}^{r-2}\frac{1}{t-i}}{t\sum_{i=0}^{r-2}\frac{1}{t-i}} \right) -  \frac{r}{f_r(e)}\frac{\sum_{i=1}^{r-1}\frac{1}{u-i}}{u\sum_{i=0}^{r-1}\frac{1}{u-i}}\\[2mm]
&= \frac{x^{1/(r-1)}}{f_r(e)^{r/(r-1)}}\left[\frac{r}{t}-\frac{r-1}
{t^2\sum_{i=0}^{r-2}\frac{1}{t-i}}\right]
- \frac{1}{f_r(e)}\left[\frac{r}{u}-\frac{r}{u^2\sum_{i=0}^{r-1}\frac{1}{u-i}}\right]. 
 \end{align*}
 When $x=f_r(e)$, by equations $t=u=s-1$, we replace $u$ by $t$  for convenience.
\begin{align*}
F'(f_r(e))
& = \frac{f_r(e)^{1/(r-1)}}{f_r(e)^{r/(r-1)}}\left[\frac{r}{t}-\frac{r-1}
{t^2\sum_{i=0}^{r-2}\frac{1}{t-i}}\right]
- \frac{1}{f_r(e)}\left[\frac{r}{t}-\frac{r}{t^2\sum_{i=0}^{r-1}\frac{1}{t-i}}\right]\\[2mm]
% & = \frac{f_r(e)^{1/(r-1)}}{f_r(e)^{r/(r-1)}}\left[\frac{r}{t}-\frac{r-1}
% {t^2\sum_{i=0}^{r-2}\frac{1}{t-i}}\right]
% - \frac{f_r(e)^{1/(r-1)}}{f_r(e)^{r/(r-1)}}\left[\frac{r}{t}-\frac{r}{t^2\sum_{i=0}^{r-1}\frac{1}{t-i}}\right]\\[2mm]
&=\frac{1}{f_r(e)}
\left[\frac{r}{t}-\frac{r-1}
{t^2\sum_{i=0}^{r-2}\frac{1}{t-i}}- \frac{r}{t} + \frac{r}{t^2\sum_{i=0}^{r-1}\frac{1}{t-i}}\right]\\[2mm]
&= \frac{1}{f_r(e)}
\left[\frac{1}{t^2}\left(\frac{r}{\sum_{i=0}^{r-1}\frac{1}{t-i}}-\frac{r-1}
{\sum_{i=0}^{r-2}\frac{1}{t-i}}\right)\right]\\
&=\frac{1}{f_r(e) t^2}\frac{\sum_{i=0}^{r-2}\left(\frac{1}{t-i}-\frac{1}{t-r+1}\right)}{(\sum_{i=0}^{r-2}\frac{1}{t-i})(\sum_{i=0}^{r-1}\frac{1}{t-i})}\\
&<0.
\end{align*}
% Since $\sum_{i=0}^{r-2}\frac{1}{t-i}< \sum_{i=0}^{r-2}\frac{1}{t-r+1}= \frac{r-1}{t-r+2}$, we have 
% \begin{align*}
% \frac{r}{\sum_{i=0}^{r-1}\frac{1}{t-i}}-\frac{r-1}
% {\sum_{i=0}^{r-2}\frac{1}{t-i}}
% &= \frac{r}{\sum_{i=0}^{r-2}\frac{1}{t-i}+\frac{1}{t-r+1}}-\frac{r-1}
% {\sum_{i=0}^{r-2}\frac{1}{t-i}}\\[2mm]
% &=  \frac{r\sum_{i=0}^{r-2}\frac{1}{t-i}-(r-1)\left[\sum_{i=0}^{r-2}\frac{1}{t-i}+\frac{1}{t-r+1} \right]}{\sum_{i=0}^{r-2}\frac{1}{t-i}(\sum_{i=0}^{r-2}\frac{1}{t-i}+\frac{1}{t-r+1})}\\[2mm]
% &= \frac{\sum_{i=0}^{r-2}\frac{1}{t-i}-(r-1)\frac{1}{t-r+1}}{\sum_{i=0}^{r-2}\frac{1}{t-i}(\sum_{i=0}^{r-2}\frac{1}{t-i}+\frac{1}{t-r+1})}\\
% & < 0
% \end{align*}
%Thus $F'(f_r(e))< 0$. 
Proof of item 2 is finished. 

Let us compute the second derivative. Since $x={t\choose r-1}$, we have
\begin{equation}
  \label{eq:dt}
  \frac{dt}{dx}=\frac{1}{\frac{dx}{dt}}=\frac{1}{x\sum_{i=0}^{r-2}\frac{1}{t-i}}.
\end{equation}
Similarly, from $e-x={u\choose r}$, we get
\begin{equation}
  \label{eq:du}
  \frac{du}{dx}=-\frac{1}{(e-x)\sum_{i=0}^{r-1}\frac{1}{u-i}}.
\end{equation}

To simplify the above equation, we compute the derivative of each main term separately, the derivative of first main term $x^{1/(r-1)}\left[ \frac{r}{t}- \frac{r-1}{t^2\sum_{i=0}^{r-2}\frac{1}{t-i}}\right]$ is
\begin{align*}
& \hspace*{.5cm} \left( x^{1/(r-1)}\left[ \frac{r}{t}- \frac{r-1}{t^2\sum_{i=0}^{r-2}\frac{1}{t-i}}\right]\right)'\\[2mm] % first line
& =\frac{x^{1/(r-1)-1}}{r-1} \left[ \frac{r}{t}- \frac{r-1}{t^2\sum_{i=0}^{r-2}\frac{1}{t-i}}\right] + x^{1/(r-1)} \frac{dt}{dx} \left[ -\frac{r}{t^2}+(r-1)\frac{1}{t^4(\sum_{i=0}^{r-2}\frac{1}{t-i})^2}\right.\\[2mm]
&\hspace*{.5cm}\left. \left(2t\sum_{i=0}^{r-2}\frac{1}{t-i} -
 t^2\sum_{i=0}^{r-2}\frac{1}{(t-i)^2} \right)\right]\\[2mm]
 % 2nd line
%  &= \frac{rx^{1/(r-1)-1}}{t(r-1)} - \frac{x^{1/(r-1)-1}}{t^2 \sum_{i=0}^{r-2}\frac{1}{t-i}} -\frac{rx^{1/(r-1)}}{t^2x\sum_{i=0}^{r-2}\frac{1}{t-i}} + \frac{x^{1/(r-1)(r-1)2t}\sum_{i=0}^{r-2}\frac{1}{t-i}}{t^6(\sum_{i=0}^{r-2}\frac{1}{t-i})^2x\sum_{i=0}^{r-2}\frac{1}{t-i}}
%  \\[2mm]
% &\hspace*{.5cm} -  \frac{x^{1/(r-1)(r-1)t^2 }\sum_{i=0}^{r-2}\frac{1}{(t-i)^2}}{t^4(\sum_{i=0}^{r-2}\frac{1}{(t-i)})^2x\sum_{i=0}^{r-2}\frac{1}{t-i}}
% \\[2mm] 
 & =\frac{x^{1/(r-1)-1}}{t^2}\left[ \frac{rt}{r-1}- \frac{r+1}{\sum_{i=0}^{r-2}\frac{1}{t-i}} + \frac{2t(r-1)}{t(\sum_{i=0}^{r-2}\frac{1}{t-i})^2}-\frac{(r-1)\sum_{i=0}^{r-2}\frac{1}{(t-i)^2}}{(\sum_{i=0}^{r-2}\frac{1}{t-i})^3}\right].
\end{align*}
Similar work for the derivative of the second main term $\frac{1}{u}-\frac{1}{u^2\sum_{i=0}^{r-1}\frac{1}{u-i}}$, we have: 
\begin{align*}
& \hspace*{.5cm} \left(\frac{1}{u}-\frac{1}{u^2\sum_{i=0}^{r-1}\frac{1}{u-i}}\right)'\\[2mm]
&= -\frac{1}{ (e-x) u^2 \sum_{i=0}^{r-1}\frac{1}{u-i}}
\left[-1 +\frac{2}{u\sum_{i=0}^{r-1}\frac{1}{u-i}}-\frac{\sum_{i=0}^{r-1}\frac{1}{(u-i)^2}}{\left(\sum_{i=0}^{r-1}\frac{1}{u-i}\right)^2}\right].
\end{align*} 

After simplification, we have
\begin{align*}
F''(x)&= \frac{x^{1/(r-1)-1}}{t^2 f_r(e)^{r/(r-1)}}\left[\frac{rt}{r-1}
-\frac{r+1}{\sum_{i=0}^{r-2}\frac{1}{t-i}}
+\frac{2(r-1)}{t\left(\sum_{i=0}^{r-2}\frac{1}{t-i}\right)^2} -
\frac{(r-1)\sum_{i=0}^{r-2}\frac{1}{(t-i)^2}}{\left(\sum_{i=0}^{r-2}\frac{1}{t-i}\right)^3}\right]\\
&\hspace*{5mm} + \frac{1}{f_r(e) (e-x) u^2 \sum_{i=0}^{r-1}\frac{1}{u-i}}
\left[-r +\frac{2r}{u\sum_{i=0}^{r-1}\frac{1}{u-i}}-\frac{r\sum_{i=0}^{r-1}\frac{1}{(u-i)^2}}
{\left(\sum_{i=0}^{r-1}\frac{1}{u-i}\right)^2}\right].
\end{align*}

Applying these two inequalities,  
\begin{align*}
\sum_{i=0}^{r-2}\frac{1}{(t-i)^2} \geq \frac{(r-1)}{t^2} \hspace*{1cm}\text{and} \hspace*{1cm}
\sum_{i=0}^{r-1}\frac{1}{(u-i)^2} \geq \frac{r}{u^2},  
\end{align*}

we have 
\begin{align*}
&
  -\frac{1}{\sum_{i=0}^{r-2}\frac{1}{t-i}}
+\frac{2(r-1)}{t\left(\sum_{i=0}^{r-2}\frac{1}{t-i}\right)^2} -
\frac{(r-1)\sum_{i=0}^{r-2}\frac{1}{(t-i)^2}}{\left(\sum_{i=0}^{r-2}\frac{1}{t-i}\right)^3} \\ 
& = -\frac{1}{\sum_{i=0}^{r-2}\frac{1}{t-i}}
\left[ 1-\frac{2(r-1)}{t\sum_{i=0}^{r-2}\frac{1}{t-i}}+ \frac{(r-1)\sum_{i=0}^{r-2}\frac{1}{(t-i)^2}}{(\sum_{i=0}^{r-2}\frac{1}{t-i})^2} \right]\\
& \leq -\frac{1}{\sum_{i=0}^{r-2}\frac{1}{t-i}}
\left[ 1-\frac{2(r-1)}{t\sum_{i=0}^{r-2}\frac{1}{t-i}}+ \frac{(r-1)^2}{t^2(\sum_{i=0}^{r-2}\frac{1}{t-i})^2} \right]\\
&= -\frac{1}{\sum_{i=0}^{r-2}\frac{1}{t-i}}
\left[ 
1- \frac{r-1}{t\sum_{i=0}^{r-2}\frac{1}{t-i}}
\right]^2\\
&\leq 0, 
\end{align*}
and similarly
\begin{equation*}
  -1 +\frac{2r}{u\sum_{i=0}^{r-1}\frac{1}{u-i}}-\frac{r\sum_{i=0}^{r-1}\frac{1}{(u-i)^2}}{\left(\sum_{i=0}^{r-1}\frac{1}{u-i}\right)^2}
\leq - \left[ 
1- \frac{r}{u\sum_{i=0}^{r-1}\frac{1}{u-i}}
\right]^2\leq 0.
\end{equation*}

Thus, we have
\begin{align*}
  F''(x)&\leq \frac{x^{1/(r-1)-1}}{t^2 f_r(e)^{r/(r-1)}}\left[\frac{rt}{r-1}
-\frac{r}{\sum_{i=0}^{r-2}\frac{1}{t-i}}\right] - \frac{r-1}{f_r(e) (e-x) u^2 \sum_{i=0}^{r-1}\frac{1}{u-i}}\\
&= \frac{x^{1/(r-1)-1}}{t^2 f_r(e)^{r/(r-1)}} \frac{r\sum_{i=0}^{r-2}\frac{i}{t-i}}{(r-1)\sum_{i=0}^{r-2}\frac{1}{t-i}} - \frac{r-1}{f_r(e) (e-x) u^2 \sum_{i=0}^{r-1}\frac{1}{u-i}}.
\end{align*}
To show the right side is negative, it is sufficient to prove
\begin{equation}
  \label{eq:f21}
  \frac{x^{1/(r-1)-1}}{t^2 f_r(e)^{1/(r-1)}} \frac{r\sum_{i=0}^{r-2}\frac{i}{t-i}}{(r-1)\sum_{i=0}^{r-2}\frac{1}{t-i}} \leq \frac{r-1}{ (e-x) u^2 \sum_{i=0}^{r-1}\frac{1}{u-i}}.
\end{equation}
Equivalently,
\begin{equation}
  \label{eq:f22}
  \frac{u^2}{t^2} \cdot \frac{e-x}{x^{1-1/(r-1)} f_r(e)^{1/(r-1)}} \cdot
\frac{r}{(r-1)^2} \cdot
\frac{\sum_{i=0}^{r-2}\frac{i}{t-i}}{\sum_{i=0}^{r-2}\frac{1}{t-i}}
\cdot \sum_{i=0}^{r-1}\frac{1}{u-i}\leq 1.
\end{equation}
Since $\frac{\sum_{i=0}^{r-2}\frac{i}{t-i}}{\sum_{i=0}^{r-2}\frac{1}{t-i}}\leq r-2$
and $\sum_{i=0}^{r-1}\frac{1}{u-i}\leq \frac{r}{u-r+1}$, it is sufficient to prove
\begin{equation}
  \label{eq:f23}
  \frac{u^2}{t^2}\cdot  \frac{e-x}{x^{1-1/(r-1)} f_r(e)^{1/(r-1)}} \cdot
\frac{r(r-2)}{(r-1)^2} 
\cdot \frac{r}{u-r+1}
\leq 1.
\end{equation}
Replacing $x={t\choose r-1}$, $e-x={u\choose r}={u\choose r-1}\frac{u-r+1}{r}
$,
and $f_r(e)={s-1\choose r-1}$, the left side of Equation \eqref{eq:f23} becomes
\begin{align*}
LHS &=  \frac{u^2}{t^2}\cdot\frac{x^{1/(r-1)}}{f_r(e)^{1/(r-1)}} \cdot
\frac{e-x}{x}\cdot
\frac{r(r-2)}{(r-1)^2} 
\cdot \frac{r}{u-r+1} \\
&= \frac{u^2}{t^2}\cdot \frac{{u\choose r-1}}{{s-1\choose t-1}^{1/(r-1)}
{t\choose r-1}^{r/(r-1)}} \frac{r^2(r-2)}{(r-1)^2}\\
&\leq \frac{r(r-2)}{(r-1)^2}\\
&< 1.
\end{align*}
The second last inequality is due to the fact (\ref{eq:tsu})  that $t\geq s-1\geq u$.
Thus, $F''(x)\leq 0$. Proof of item 3 is finished.

\end{proof}

\section{Proof of main theorem}

\begin{proof}[Proof of Theorem \ref{Th1}] 
We will use double inductions on $r$ and $e$ to prove the theorem. 
For $r = 2$ and any $e\geq 0$, Theorem \ref{Th1} is just Stanley's theorem.

Inductively, we assume the statement is true for all $(r-1)$-hypergraphs. For $r$-hypergraph, clearly, the statement is trivial for the cases $e=0, 1$. We assume the statement holds for all
$r$-hypergraphs with less than $e$ edges.

Let $H$ be the maximum hypergraph among all $r$-hypergraphs of $e$ edges.
By Lemma \ref{makelarger}, $H$ has only one non-trivial connected component. By deleting isolated
vertices if possible, we may assume that $H$ is connected. 

Let $\textbf{x}$ be the Perron-Frobenius vector of $H$ and 
$v$ be a vertex such that $x_v = \max\{ x_u, u\in V(H)\}$.
By Lemma \ref{GvHv},  the degree $d$ of $v$ is at least $f_r(e)$. 
Recall that $H_v$ is the induced hypergraph obtained from $H$ by deleting the vertex $v$ and
$G_v$ is the link graph of $H$ at $v$.  

The main idea is to construct an $\alpha$-subnormal labeling for $H$ by combining
the $\alpha_1$-normal labeling of $G_v$ and the $\alpha_2$-normal labeling of $H_v$ properly.
By Lemma \ref{GvHv}, $G_v$ is a connected $(r-1)$-hypergraph with $d$ edges.
By inductive hypothesis, we have
\begin{equation}
  \label{eq:Gv}
  \rho(G_v)\leq f_{r-1}(d).
\end{equation}

By Lemma \ref{connectedgraph}, $G_v$ has a consistent $\alpha_1$-normal labeling with
$\alpha_1=\rho(G_v)^{-(r-1)}$. Let $B_1$ be the  weighted incidence matrix of $G_v$
 corresponding to this $\alpha_1$-normal labeling. We have
 \begin{align}
   \label{eq:B1v}
 \sum\limits_{f \in E(G_v)\colon u\in f}B(u, f)&=1, \quad \mbox{ for any vertex } u\in V(G_v),\\
\label{eq:B1e}
\prod\limits_{u \in f}B(u, f)&=\alpha_1, \quad \mbox{ for any edge } f\in E(G_v).
 \end{align}

Let $H'_v$ be the unique non-trivial connected component of $H_v$. Then
$H'_v$ has $|E(H_v)|=e-d$ edges. By inductive hypothesis, we have
\begin{equation}
  \label{eq:Hv}
  \rho(H'_v)\leq f_{r}(e-d).
\end{equation}

By Lemma \ref{connectedgraph}, $H'_v$ has a consistent $\alpha_2$-normal labeling with $\alpha_2=\rho(H'_v)^{-r}$. Let $B_2$ be the  weighted incidence matrix of $H'_v$
 corresponding to this $\alpha_2$-normal labeling. We have
\begin{align}
   \label{eq:B2v}
 \sum\limits_{f \in E(H'_v)\colon u\in f}B(u, f)&=1, \quad \mbox{ for any vertex } u\in V(H'_v),\\
\label{eq:B2e}
\prod\limits_{u \in f}B(u, f)&=\alpha_2, \quad \mbox{ for any edge } f\in E(H'_v).
 \end{align}

Now we define a weighed incidence matrix $B$ of the hypergraph $H$.
For any vertex $u\in V(H)$ and any edge $f\in E(H)$, we have
\begin{equation}
  \label{eq:B}
  B(u,f)=
  \begin{cases}
    0 & \mbox{ if } u\notin f;\\
   1/d & \mbox{ else if } u=v;\\
  xB_1(u, f-\{v\}) & \mbox{ else if } u\in f;\\
  yB_2(u, f) & \mbox{ otherwise.}
  \end{cases}
\end{equation} 
Here $x$, $y$ are two real numbers in $[0,1]$ and will be chosen later.

Now consider the following two properties about $B$:
\begin{itemize}
\item
For each vertex $u\in V(H)$, we estimate 
$\sum\limits_{f \in E(H), u\in f}B(u, f)$ as follows:

If $u=v$,
we have
\begin{equation}\label{eq:Bv}
\sum\limits_{f \in E(H), v\in f}B(v, f)=d\times \frac{1}{d}=1. 
\end{equation}
If $u\not= v$, we have
\begin{equation}\label{xiaddyj} 
\sum\limits_{f \in E(H), u\in f}B(u, f)=
x\sum\limits_{f' \in E(G_v), u\in f'}B(u, f')
+ y \sum\limits_{f \in E(H_v), u\in f}B(u, f) 
\leq x + y.
\end{equation}
Here we applied Equations \eqref{eq:B1v} and \eqref{eq:B2v}. Notice that
if $u$ is an isolated vertex of $H_v$, then the second sum is $0$. Nevertheless,
the above inequlity holds.

\item Now we estimate $\prod\limits_{u\in f} B(u,f)$ for each edge $f\in E(H)$.

If $v\in f$, then
\begin{equation}\label{alpha}
\prod\limits_{u \in f}B(u, e)=\frac{1}{d}\prod\limits_{u \in f-\{v\}}xB_1(u, e')= \frac{1}{d}x^{r-1}\alpha_1.
\end{equation}
If $v\not\in f$, then
\begin{equation}\label{alpha1}
 \prod\limits_{u \in f}B(u, e)=\prod\limits_{u\in f}yB_2(u, e)= y^r\alpha_2. 
\end{equation}

Set $\alpha=\left(f_r(e)\right)^{-r}$, $x=\left(\frac{d \alpha }{\alpha_1}\right)^{1/(r-1)}$,
and $y=\left(\frac{\alpha}{\alpha_2}\right)^{1/r}$. Then for each $f\in E(H)$,
we have
\begin{equation}
  \label{eq:Bf}
  \prod\limits_{u \in f}B(u, e)=\alpha.
\end{equation}

Recall $\alpha_1=\rho(G_v)^{-(r-1)}$ and $\alpha_2=\rho(H'_v)^{-r}$.
Combining with Inequalities \eqref{eq:Gv} and \eqref{eq:Hv}, we have
\begin{align*}
x + y
&=\left(\frac{d \alpha }{\alpha_1}\right)^{1/(r-1)} + \left(\frac{\alpha}{\alpha_2}\right)^{1/r}\\
& =\frac{d^{\frac{1}{r-1}} \rho(G_v)}{f_r(e)^{\frac{r}{r-1}}} + \frac{\rho(H_v)}{f_r(e)} \\
&\leq \frac{d^{\frac{1}{r-1}}f_{r-1}(d)}{f_r(e)^{\frac{r}{r-1}}}  + \frac{f_r(e-d)}{f_r(e)} \\
& \leq 1. 
\end{align*}
The last inequality is due to Lemma \ref{l2}
 since $f_r(e)\leq d\leq e$. 
\end{itemize}

Combining this with Equations \eqref{eq:Bv} and (\ref{xiaddyj}), we have
\begin{equation}
  \label{eq:Bu}
\sum\limits_{f \in E(H), u\in f}B(u, f)\leq 1.
\end{equation}
Equations \eqref{eq:Bf} and \eqref{eq:Bu} imply that
$H$  is  $\alpha$-subnormal with $\alpha=\left(f_r(e)\right)^{-r}$. 
Hence, by Lemma \ref{subtheorem}, we have 
$$\rho(H)\leq f_r(e).$$ 

When the inequality holds, we must have $e={k\choose r}$, $d=f_{r}(e)={k-1\choose r-1}$,
$\rho(G_v)={k-2\choose r-2}$, and $\rho(H_v)={k-2\choose r-1}$. By induction, $G_v$ is the
complete graph $K_{k-1}^{r-1}$ and $H_v$ is the complete graph $K_{k-1}^r$. Thus,
 $H$ is the complete graph $K_k^r$.
Since adding isolated vertices will not change the number of edges and the spectral radius,
 the inequality in Theorem \ref{Th1} holds if and only if $H$ is 
the complete hypergraph possibly with some isolated vertices added. 
\end{proof}


\begin{thebibliography}{1}
\bibliographystyle{plain}


\bibitem{Brualdi}  R. A. Brualdi and  A. J. Hoffman, On the spectral radius of $(0, 1)$-matrices, {\it Linear Algebra Appl.}   65 (1985) pp. 133-146.

 

\bibitem{ChangPearsonZhang} K.C. Chang, K. Pearson, and T. Zhang, Perron-Frobenius theorem for non negative tensors, {\it Commun. Math. Sci. }  6 (2008) pp. 507-520.


\bibitem{CQZ} K.C. Chang, L. Qi, and T. Zhang, A survey on the spectral theory of nonnegative tensors,
{\it Numer. Linear Algebra Appl.}   20 (2013) pp. 891-912.



\bibitem{CCZ} D. Chen, Z. Chen, and X. Zhang, Spectral radius of uniform hypergraphs and degree sequences, {\it Front. Math. China} (2017) doi:10.1007/s11464-017-0626-3. 


\bibitem{Cooper} J. Cooper and A. Dutle, Spectra of uniform hypergraphs, {\it Linear Algebra Appl.} 436
(2012) pp. 3268-3292.


\bibitem{FTPL} Y. Fan, Y. Tan, X. Peng, and A. Liu, Maximizing spectral radii of uniform hypergraphs with few edges, {\it Discuss. Math. Graph Theory}  36(4) (2016) pp. 854-856. 

\bibitem{Friedland} S. Friedland, Bounds on the spectral radius of graphs with $e$ edges, {\it Linear Algebra Appl.}  101 (1988) pp. 81-86. 


\bibitem{FriedlandGaubetHan} S. Friedlanda, S. Gaubert, and L. Han, 
Perron-Frobenius theorem for nonnegative multilinear forms and extensions, 
{\it Linear Algebra Appl.}  438 (2013) pp. 738-749.

\bibitem{KLQY} L. Kang, L. Liu, L. Qi, and X. Yuan, Some results on the spectral radii of uniform hypergraphs, arXiv:1605.01750 [math.CO].


\bibitem{Katona} Gyula O. H. Katona, A theorem of finite sets.  In {\it Theory of Graphs: Proceedings}  Academic Press (1968) pp. 187-207.


\bibitem{Kruskal} J. Kruskal, The optimal number of simplices in a complex. {\it Math. Opt. Techniques} (1963) pp. 251-268.

\bibitem{PerasonZhang} K. Pearson and  T. Zhang, On spectral hypergraph theory of the adjacency tensor, 
{\it Graphs Combin.} 30 (2014) 1233-1248. 

\bibitem{LSQ} H. Li, J. Shao, and L. Qi, The extremal spectral radii of $k$-uniform supertrees,  {\it Journal of Combinatorial Optimization}   32(3)  pp. 741-764.


\bibitem{Lovas}  L. Lov\'asz, Combinatorial problems and exercises. {\it North-Holland Publ.,} Amsterdam, (1979). 


\bibitem{LuMan} L. Lu and S. Man, Connected hypergraphs with small spectral radius, {\it Linear Algebra Appl.} 509  (2016)  pp. 206-227.


\bibitem{Qi} L. Qi, Eigenvalues of a real supersymmetric tensor, {\it Journal of Symbolic Computation} 40 (2005) pp. 1302-1324. 

\bibitem{Row} P. Rowlinson, On the maximal index of graphs with a prescribed number of edges, {\it Linear Algebra Appl.}  110 (1988) pp. 43-53. 

\bibitem{stanley} R.P. Stanley,   A bound on the spectral radius of graphs with $e$ edges. {\it Linear Algebra Appl.}  67 (1987) pp. 267-269.

\bibitem{XWL} P. Xiao, L. Wang, and Y. Lu, The maximum spectral radii of uniform supertrees
with given degree sequences, {\it Linear Algebra Appl.}  523 (2017) pp. 33-45. 


\bibitem{YangYang} Y. Yang and  Q. Yang, Further results for Perron-Frobenius theorem for nonnegative tensors, {\it SIAM J. Matrix Anal. Appl.} 31 (2010) 2517-2530.


\bibitem{ZLKB}
W. Zhang, L. Liu, L. Kang, and Y. Bai, Some properties of the spectral radius for general hypergraphs, {\it Linear Algebra Appl.}
 513  pp. 103-119.



\end{thebibliography}
\end{document}